\documentclass[preprint,12pt,authoryear]{elsarticle}

\usepackage{amssymb}

\usepackage{amsthm,amsmath}
\newtheorem{thm}{Theorem}[section]
\newtheorem{lem}[thm]{Lemma}
\newtheorem{rem}[thm]{Remark}
\journal{arXiv.org}

\begin{document}

\begin{frontmatter}

\title{An Alternative Approach to Convolutions of Harmonic Mappings}

\author[label1]{Chinu Singla}
\author[label2]{Sushma Gupta}
\author[label3]{Sukhjit Singh}
\address[label1]{Sant Longowal Institute of Engineering and Technology,
Longowal, Punjab, India.\\ $chinusingla204@gmail.com$}
\address[label2]{Sant Longowal Institute of Engineering and Technology,
Longowal, Punjab, India.\\ $sushmagupta1@yahoo.com$}
\address[label3]{Sant Longowal Institute of Engineering and Technology,
Longowal, Punjab, India.\\ $sukhjit\_d@yahoo.com$}              

\begin{abstract}
Convolutions or Hadamard products of analytic functions is a well explored area of research and many nice results are available in literature. On the other hand, very little is known in general about the convolutions of univalent harmonic mappings. So, researchers started exploring properties of convolutions of some specific univalent harmonic mappings and while doing so, they have mostly used well known \textit{\lq Cohn's rule\rq} or/and \textit{\lq Schur-Cohn's algorithm\rq}, which involves computations that are very cumbersome. The main objective of this article is to present an alternative approach, which requires very less computational efforts and allows us to prove more general results. Most of the earlier known results follow as particular cases of the results proved herein.
\end{abstract}


\end{frontmatter}

\section{Introduction} \label{sec:1}
Let us condider a class $S_H$ of complex valued univalent harmonic functions $f$ in the open unit disc $\mathbb{D}=\{z \in \mathbb{C}:|z|<1\},$ having canonical representation $f=h+\overline{g}$ and normalized by the conditions $h(0)=g(0)=0=h'(0)-1.$  The Jacobian $J_f(z)$ of $f$ is given by $$J_f(z)=|f_z(z)|^2 -|f_{\overline{z}}(z)|^2=|h'(z)|^2-|g'(z)|^2.$$ \cite{1} proved that a harmonic function $f=h+\overline{g}$ is locally univalent and sense preserving in $\mathbb{D}$ if and only if $J_f(z)>0$ in $\mathbb{D}.$ This is equivalent to the existence of an analytic function $w(z)=\frac{g'(z)}{h'(z)},$ satisfying $|w(z)|<1$ in $\mathbb{D}.$ Here, $w$ is called the dilatation of $f$ (see \cite{2}).
\par Denote by $S_H^0$ the subclass of $S_H$ having functions $f$ with additional normalization condition $f_{\overline{z}}(0)=0.$ The subclass of $S_H(S_H^0)$ containing convex functions is denoted by $K_H(K_H^0).$ Let $S\subset S_H$ be the class of analytic and univalent functions and let $K, S^*$ and $C$ be the usual subclasses of $S$ containing convex, starlike and close-to-convex functions, respectively. A domain $E$ in $\mathbb{C}$ is said to be convex in the direction $\psi, 0 \leq \psi < \pi,$ if every line parallel to the line through $0$ and $ e ^{i \psi}$ has an empty or connected intersection with $E$. A function $f$ is said to be convex in the direction of $\psi$ if it maps $\mathbb{D}$ onto the domain convex in the direction of $\psi, 0 \leq \psi < \pi.$ If $\psi=0,$ then $f$ is said to be convex in the direction of the real axis and if $\psi=\pi/2,$ then $f$ is said to be convex in the direction of the imaginary axis.
\par The convolution or Hadamard product of two analytic functions $f(z)=z+\sum_{n=2}^\infty a_n z^n$ and $F(z)=z+\sum_{n=2}^\infty A_n z^n,$  is denoted by $f\ast F$ and is defined as $(f \ast F)(z)=z+ \sum_{n=2}^\infty a_nA_nz^n$. Let $f=h+\overline{g}$ and $F=H+\overline{G}$ be two harmonic mappings, then their convolution is denoted by $f\tilde{\ast} F$ and is defined as $f\tilde{\ast} F=h \ast H +\overline{g\ast G}.$
Many pleasant results are available in literature on the convolution of analytic functions. For example, \cite{22} proved that \begin{enumerate}
\item $f\ast g \in K$ for all $f,g \in K;$
\item $f\ast g \in S^*$ for all $f\in K$ and $g\in S^*;$
\item $f\ast g \in C$ for all $f\in K$ and $g\in C.$
\end{enumerate} 
On the other hand, not much is known about the properties of convolutions of harmonic functions and there are no general results of the kind stated above in this case. For example, for $F,G \in K_H,$ $F\tilde{\ast} G$ may not be univalent in $\mathbb{D}$ even (see \cite{13}). However, some researchers started exploring the nature of convolutions of some specific harmonic maps. We mentioned some of them below.
\par It is well known that if $f=h+\overline{g} \in S_H^0$ maps $\mathbb{D}$ onto the right half plane $R=\{w\in \mathbb{C} :Re w>-1/2\}, $ then it must satisfy $$ h(z)+g(z)=\frac{z}{1-z}.$$ In 2001, \cite{5} started study of convolution of right half plane harmonic mappings.  He presented following results.   
\begin{thm} Let $f_i=h_i+\overline{g_i}$ be the harmonic right half plane mappings with $h_i(z)+g_i(z)=\frac{z}{1-z}$ for $i=1,2.$ Then $f_1 \tilde{\ast} f_2\in S_H^0$ and is convex in the direction of the real axis provided, $f_1 \tilde{\ast} f_2$ is locally univalent and sense preserving in $\mathbb{D}.$\end{thm}
In the same paper, he defined a family of harmonic mappings $V_{\beta}=u_{\beta} + \overline{v_{\beta}},$ obtained from analytic strip mappings, by using \textit{\lq shearing technique \rq} \cite{4} \begin{equation} \label{eq:3}
u_{\beta}(z) + v_{\beta}(z)=\frac{1}{2iSin\beta}\log\left(\frac{1+ze^{i\beta}}{1+ze^{-i\beta}}\right),0<\beta<\pi
\end{equation} and established the following result.
\begin{thm} Let $f=h+\overline{g}$ be the harmonic right half plane mappings with $h(z)+g(z)=\frac{z}{1-z}$ and $V_{\beta}=u_{\beta} + \overline{v_{\beta}}\in S_H$ as defined in \eqref{eq:3} with $\pi/2 \leq \beta<\pi$. Then $f \tilde{\ast} V_{\beta}\in S_H^0$ and is convex in the direction of the real axis provided, $f \tilde{\ast} V_{\beta}$ is locally univalent and sense preserving in $\mathbb{D}.$\end{thm}
\noindent Let $F_0=H_0+\overline{G_0}$ be the harmonic right half plane mapping given by  \begin{equation} \label{eq:standard}
H_0+G_0 =z/(1-z),\hspace{2cm} G_0'/H_0'=-z. 
\end{equation}  Then $F_0$ is called the standard right half plane mapping. \cite{3} were able to drop the requirement of local univalence and sense preserving from the above results as under.
\begin{thm} Let $F=H+\overline{G}$ be the harmonic right half plane mapping with $H(z)+G(z)=\frac{z}{1-z}$ and $G'(z)/H'(z)=e^{i\theta} z^n(\theta \in \mathbb{R},n\in \mathbb{N}).$ Then for $n=1,2,$ $F_0 \tilde{\ast} F\in S_H^0$ and is convex in the direction of the real axis. Here, $F_0$ is the standard right half plane mapping as defined above.\end{thm}
\begin{thm} Let $V_{\beta}=u_{\beta} + \overline{v_{\beta}}\in S_H$ be the harmonic mapping as defined in \eqref{eq:3} with $v_{\beta}'(z)/u_{\beta}'(z)=e^{i\theta} z^n(\theta \in \mathbb{R},n\in \mathbb{N}).$ Then for $n=1,2,$ $F_0 \tilde{\ast} V_{\beta}\in S_H^0$ and is convex in the direction of the real axis.\end{thm}
For a real number $\gamma,0 \leq \gamma <2\pi,$ the mappings $f_\gamma=h_\gamma+\overline{g_\gamma},$ given by \begin{equation} \label{eq:5}
 h_\gamma(z)+e^{-2i\gamma} g_\gamma(z)=\frac{z}{1-e^{i\gamma}z}
\end{equation} is called slanted right half plane mapping and maps the unit disc $\mathbb{D}$ onto slanted right half plane given by $H(\gamma)=\{w\in \mathbb{C} :Re (e^{i\gamma} w)>-1/2\}$ (see \cite{3}). We denote by $S_{H(\gamma)}$, the family of all slanted right half plane mappings.\\
 \cite{35} proved the following.
\begin{thm} Let $f_\gamma \in S_{H(\gamma)}$ be as defined above with dilatation $g_\gamma'(z)/h_\gamma'(z)=e^{i\theta}z^n,\theta \in \mathbb{R},n=1,2.$ Then $F_0\tilde{\ast} f_\gamma \in S_H^0$ and is convex in the direction of $-\gamma.$  
\end{thm}
 \cite{7} generalized Theorem 1.3 by taking $F_a$ instead of $F_0,$ where $F_a=H_a +\overline{ G_a}\in K_H$ is the right half plane mapping given by
\begin{equation} \label{eq:2}
H_a(z)+G_a(z)=\frac{z}{1-z}, \hspace{0.5cm}
\frac{G_a'(z)}{H_a'(z)}=\frac{a-z}{1-az}, a\in (-1,1) 
\end{equation}
and proved the following.
\begin{thm} Let $F=H+\overline{G}$ be the harmonic right half plane mapping with $H(z)+G(z)=\frac{z}{1-z}$ and $G'(z)/H'(z)=e^{i\theta} z^n(\theta \in \mathbb{R},n\in \mathbb{N}).$ Then $F_a \tilde{\ast} F\in S_H$ and is convex in the direction of the real axis for all $a\in [(n-2)/(n+2),1).$\end{thm}
Similarly, in an attempt to generalize Theorem 1.4, \cite{6} proposed the following conjecture. Which they themselves proved for $0<\beta <\pi$ and $n=1,2,3,4;$ and also for the case when $\beta =\pi/2$ and $n$ is a natural number. \\ 
\begin{thm} Let $V_{\beta}=u_{\beta} + \overline{v_{\beta}}\in S_H$ be the harmonic mapping as defined in \eqref{eq:3} with $v_{\beta}'(z)/u_{\beta}'(z)=e^{i\theta} z^n(\theta \in \mathbb{R},n\in \mathbb{N}).$ Then $F_a \tilde{\ast} V_{\beta}\in S_H$ and is convex in the direction of the real axis for all $a\in [(n-2)/(n+2),1),$ where, $F_a$ is the harmonic right half plane mapping as defined in \eqref{eq:2}.\end{thm}
 \cite{12} gave the alternative proof of the above result for all $n \in \mathbb{N}$ and $\beta=\pi/2.$\\
Recently, \cite{10} considered more general slanted right half plane mappings $F_{(a,\alpha)}=H_{(a,\alpha)} +\overline{ G_{(a,\alpha)}}\in S_{H(\alpha)},$ essentially given by 
\begin{equation} \label{eq:4}
H_{(a,\alpha)}(z)+e^{-2i\alpha}G_{(a,\alpha)}(z)=\frac{z}{1-ze^{i\alpha}}, \hspace{0.3cm}
\frac{G_{(a,\alpha)}'(z)}{H_{(a,\alpha)}'(z)}=e^{2i\alpha}\left(\frac{a-ze^{i\alpha}}{1-aze^{i\alpha}}\right),
\end{equation} 
where, $a \in(-1,1),\alpha \in [0,2\pi)$ and studied its convolution with another slanted right half plane mapping $f_\gamma=h_\gamma+\overline{g_\gamma}$ given by \eqref{eq:5} and with dilatation \begin{equation} \label{dilatation}
g_\gamma'(z)/h_\gamma'(z)=e^{i\theta}z^n,\theta \in \mathbb{R},n\in \mathbb{N}.
\end{equation} They obtained the following result.
\begin{thm} Let $F_{(a,\alpha)}=H_{(a,\alpha)} +\overline{ G_{(a,\alpha)}}\in S_{H(\alpha)}$ be given by \eqref{eq:4} and $f_\gamma=h_\gamma+\overline{g_\gamma}$ be given by \eqref{eq:5} with dilatation as given in \eqref{dilatation}. Then $F_{(a,\alpha)} \tilde{\ast} f_{\gamma}\in S_H$ and is convex in the direction of $-(\alpha+ \gamma)$ for all $a\in [(n-2)/(n+2),1).$\end{thm}
  \cite{11} considered a new family of harmonic mappings $f_c=h_c+\overline{g_c},$ convex in horizontal direction given by 
\begin{equation} \label{eq:6}
h_c(z)-g_c(z)= \frac{z}{1-z} \hspace{0.5cm}
\frac{g_c'(z)}{h_c'(z)}=z.
\end{equation}
and presented following result.
\begin{thm} For $n \in \mathbb{N},$ let $f_n=h_n+\overline{g_n},$ be harmonic mappings with $$h_n(z)-g_n(z)=\frac{1}{2i \sin \psi}\log\left(\frac{1+ze^{i\psi}}{1+ze^{-i\psi}}\right),\pi/2 \leq \psi<\pi$$ and $\frac{g_n'(z)}{h_n'(z)}=e^{i\theta} z^n(\theta \in \mathbb{R}).$ If $n=1,2,$ then $f_c \tilde{\ast} f_n \in S_H^0$ and is convex in the direction of the real axis, where $f_c$ is given by \eqref{eq:6}.\end{thm}
It has been observed that most of the results listed out above have been proved mostly by using \textit{\lq Cohn's rule\rq} or/and \textit{\lq Schur-Cohn's algorithm\rq} \cite{7} and computations involved are extremely cumbersome. The main objective of the present article is to present a technique, which is simple to apply and involves very less computations. Our technique enables us to prove more general results and all the results stated above, deduce as particular cases of the results obtained herein.
\section{Convolution of Convex Harmonic Functions}  
We begin this section by proving the following simple result.
\begin{lem}
Let $k$ and $k'$ be real numbers with $k'-k>0.$ Then for $w\in \mathbb{C}$, $$\left|\frac{k+w}{k'+w}\right|<1$$ if and only if $$Re(w)>-\left(\frac{k+k'}{2}\right).$$
\end{lem}
\begin{proof}
We can easily see that $$\left|\frac{k+w}{k'+w}\right|<1 $$ if and only if
$$k^2+|w|^2+2k Re(w)<k'^2+|w|^2+2k'Re(w).$$ This is equivalent to $$Re(w)>-\left(\frac{k+k'}{2}\right)$$ as $k'-k>0.$
\end{proof}
We shall also need the following result whose proof runs on the same lines as that of Theorem 2 in \cite{3} and hence is omitted.
\begin{lem} Let $F_1=H_1 +\overline{ G_1}\in S_{H(\alpha)}, F_2=H_2+\overline{G_2} \in S_H$ be two harmonic functions with $H_1(z)+e^{-2i\alpha}G_1(z)=\frac{z}{1-ze^{i\alpha}}$ and $H_2(z)+ e^{-2i \gamma} G_2(z)=f(z),$ where $$zf'(z)= \frac{z}{(1+ze^{i(\eta + \gamma)})(1+ze^{-i(\eta - \gamma)})}$$ for some $\eta \in \mathbb{R}.$ Then $F_1 \tilde{\ast} F_2 \in S_H$ and is convex in the direction of $-(\alpha+\gamma),$ provided  $F_1 \tilde{\ast} F_2$ is locally univalent and sense preserving in $\mathbb{D}.$
\end{lem}
Now, consider a family of harmonic mappings $T_{(\eta,\gamma)}=R_{(\eta,\gamma)}+\overline{S_{(\eta,\gamma)}},$ given by
\begin{equation} \label{eq:7}
R_{(\eta,\gamma)}(z)+ e^{-2i \gamma} S_{(\eta,\gamma)}(z)=f(z),\hspace{1cm} \frac{S_{(\eta,\gamma)}'(z)}{R_{(\eta,\gamma)}'(z)}=e^{i\theta} z^n,
\end{equation}
where $f$ is the analytic mapping in $\mathbb{D}$ given by $$zf'(z)= \frac{z}{(1+ze^{i(\eta + \gamma)})(1+ze^{-i(\eta - \gamma)})}$$ for some $\eta \in \mathbb{R}.$ Pommerenke \cite{9} has shown that $zf'$ is starlike in $\mathbb{D}.$

\begin{thm}
Let $F_{(a,\alpha)}=H_{(a,\alpha)} +\overline{ G_{(a,\alpha)}} \in S_{H(\alpha)}$ be given by \eqref{eq:4}. Then $F_{(a,\alpha)} \tilde{\ast} T_{(\eta,\gamma)} \in S_H$ and is convex in the direction of $-(\alpha+\gamma)$ for  $a\in \left[\frac{n-2}{n+2},1\right).$ Here,  $T_{(\eta,\gamma)}\in S_H$ is given by \eqref{eq:7}. 
\end{thm}
\begin{proof} From equations in \eqref{eq:4}, using \textit{\lq shear technique\rq} (see \cite{4}), we get $$H_{(a,\alpha)}(z)=\frac{1}{2}\left[\frac{z}{(1-ze^{i\alpha})}+\left(\frac{1-a}{1+a}\right)\frac{z}{(1-ze^{i\alpha})^2}\right]$$ and $$G_{(a,\alpha)}(z)=\frac{1}{2}\left[\frac{ze^{2i\alpha}}{(1-ze^{i\alpha})}-\left(\frac{1-a}{1+a}\right)\frac{ze^{2i\alpha}}{(1-ze^{i\alpha})^2}\right].$$ In view of Lemma 2.2, it is enough to prove that dilatation $W(z)=\displaystyle\frac{( G_{(a,\alpha)} \ast S_{(\eta,\gamma)})'(z)}{(H_{(a,\alpha)} \ast R_{(\eta,\gamma)})'(z)}$ of $F_{(a,\alpha)}\tilde{\ast} T_{(\eta,\gamma)}$ satisfies $|W(z)|<1$ in $\mathbb{D}.$ As $$W(z)=e^{2i\alpha} \displaystyle\frac{( G_{(a,0)} \ast S_{(\eta,\gamma)})'(ze^{i\alpha})}{(H_{(a,0)} \ast R_{(\eta,\gamma)})'(ze^{i\alpha})}=e^{2i\alpha}\widehat{w}(ze^{i\alpha})(say),$$ it is therefore enough to prove that $|\widehat{w}(z)|<1$ in $\mathbb{D}.$ Further, note that \begin{equation} \label{eq:8}
\widehat{w}(z)=\frac{2aS_{(\eta,\gamma)}'(z)-(1-a)zS_{(\eta,\gamma)}''(z)}{2R_{(\eta,\gamma)}'(z)+(1-a)zR_{(\eta,\gamma)}''(z)}.
\end{equation}
From \eqref{eq:7}, we get $$S_{(\eta,\gamma)}'(z)=e^{i\theta}z^n R_{(\eta,\gamma)}'(z)$$ and so $$S_{(\eta,\gamma)}''(z)=e^{i\theta}z^n R_{(\eta,\gamma)}''(z)+n e^{i\theta}z^{n-1} R_{(\eta,\gamma)}'(z).$$ Putting these values of $S_{(\eta,\gamma)}'$ and $S_{(\eta,\gamma)}''$ in equation \eqref{eq:8}, we have \begin{equation} \label{eq:9}
\widehat{w}(z)=-e^{i\theta}z^n\displaystyle \left[\frac{\frac{n-(n+2)a}{1-a}+\frac{zR_{(\eta,\gamma)}''(z)}{R_{(\eta,\gamma)}'(z)}}{\frac{2}{1-a}+\frac{zR_{(\eta,\gamma)}''(z)}{R_{(\eta,\gamma)}'(z)}}\right].
\end{equation}
Now, $|\widehat{w}(z)|<1$ if  \begin{equation} \label{eq:10}
\displaystyle\left|\frac{\frac{n-(n+2)a}{1-a}+\frac{zR_{(\eta,\gamma)}''(z)}{R_{(\eta,\gamma)}'(z)}}{\frac{2}{1-a}+\frac{zR_{(\eta,\gamma)}''(z)}{R_{(\eta,\gamma)}'(z)}}\right|\leq 1. \end{equation}
For $a=\frac{n-2}{n+2},$ we note that left hand side of \eqref{eq:10} is equal to 1 and for  $a\in \left(\frac{n-2}{n+2},1\right),$ $$\left(\frac{2}{1-a}\right)-\left(\frac{n-(n+2)a}{1-a}\right)>0.$$ Therefore, in view of Lemma 2.1, it is sufficient to prove that
 \begin{equation} \label{eq:10.1}
 Re\left\{\frac{zR_{(\eta,\gamma)}''(z)}{R_{(\eta,\gamma)}'(z)}\right\}>-\frac{n+2}{2}.\end{equation}
 From \eqref{eq:7}, we have $$R_{(\eta,\gamma)}'(z)=\frac{f'(z)}{1+e^{i(\theta-2\gamma)}z^n},$$  which gives $$\frac{zR_{(\eta,\gamma)}''(z)}{R_{(\eta,\gamma)}'(z)}=\frac{zf''(z)}{f'(z)}-\frac{ne^{i(\theta-2\gamma)}z^n}{1+e^{i(\theta-2\gamma)}z^n},$$
 or equivalently $$n+2+\frac{2zR_{(\eta,\gamma)}''(z)}{R_{(\eta,\gamma)}'(z)}=2\left(1+\frac{zf''(z)}{f'(z)}\right)+n\left(\frac{1-e^{i(\theta-2\gamma)}z^n}{1+e^{i(\theta-2\gamma)}z^n}\right).$$ Now, $Re\left(1+\frac{zf''(z)}{f'(z)}\right)>0$ in $\mathbb{D}$ because $f$ is convex in $\mathbb{D}$ (as $zf'$ is starlike in $\mathbb{D}$) and for $|z|<1,$ we have $Re\left(\frac{1-e^{i(\theta-2\gamma)}z^n}{1+e^{i(\theta-2\gamma)}z^n}\right)>0.$ This gives $$Re\left\{n+2+\frac{2zR_{(\eta,\gamma)}''(z)}{R_{(\eta,\gamma)}'(z)}\right\}>0$$ in $\mathbb{D},$ which in turn, shows that \eqref{eq:10.1} is true.\end{proof}
 \begin{rem} By assigning suitable values to the parameters $\alpha,\gamma$ and $\eta$ in Theorem  2.3, we easily deduce most of the results listed in Section 1 as under.
 \begin{enumerate}
\item If we put $\eta=\pi$ in \eqref{eq:7}, we get $$R_{(\pi,\gamma)}(z)+ e^{-2i \gamma} S_{(\pi,\gamma)}(z)=\frac{z}{1-ze^{i\gamma}}$$
and Theorem 2.3 reduces to Theorem 1.8.
\item   On taking $\alpha=\gamma=0$ and $0<\eta <\pi$ in Theorem 2.3, we get Theorem 1.7. In addition, if we take $a=0$ also, then we get Theorem 1.4.
\item By taking $\alpha=a=0$ and $\eta =\pi$ in Theorem 2.3, we get Theorem 1.5.
\item Setting $\alpha=\gamma=0$ and $\eta =\pi$ in Theorem 2.3 gives Theorem 1.6,  which in turn reduce to Theorem 1.3, if we  substitute $a=0$ also.
\end{enumerate}
\end{rem}
\section{Convolution of Harmonic Functions convex in one direction}  \label{sec:4}
In this section, we study convolutions of harmonic functions convex in one direction. To proceed further, we need following result whose proof is omitted as it follows similarly as the proof of Theorem 2 in \cite{3}. 
\begin{lem}
Let $f_1=h_1 +\overline{ g_1}\in S_{H}$ and $f_2=h_2+\overline{g_2} \in S_H$ be two harmonic functions with $h_1(z)-e^{-2i\alpha}g_1(z)=\frac{z}{1-ze^{i\alpha}}, \alpha \in [0,2\pi)$ and $h_2(z)- e^{-2i \gamma} g_2(z)=f(z),$ where, $$zf'(z)= \frac{z}{(1+ze^{i(\eta + \gamma)})(1+ze^{-i(\eta + \gamma)})}$$ for some $\eta \in \mathbb{R}.$ Then $f_1 \tilde{\ast} f_2 \in S_H$ and is convex in the direction of $-(\alpha+\gamma),$ provided  $f_1 \tilde{\ast} f_2$ is locally univalent and sense preserving in $\mathbb{D}.$
\end{lem}
Now, consider harmonic mapping $f_{b,\alpha}=h_{b,\alpha} +\overline{ g_{b,\alpha}},b\in(-1,1),\alpha \in [0,2\pi)$ with 
\begin{equation} \label{eq:11}
h_{b,\alpha}(z)-e^{-2i\alpha}g_{b,\alpha}(z)=\frac{z}{1-ze^{i\alpha}}, \hspace{0.5cm}
\frac{g_{b,\alpha}'(z)}{h_{b,\alpha}'(z)}=e^{2i\alpha}\left(\frac{b+ze^{i\alpha}}{1+bze^{i\alpha}}\right)
\end{equation}
and another harmonic mapping $t_{\eta,\gamma}=r_{\eta,\gamma}+\overline{s_{\eta,\gamma}},\gamma \in [0,2\pi)$ given by 
\begin{equation} \label{eq:12}
r_{\eta,\gamma}(z)- e^{-2i \gamma} s_{\eta,\gamma}(z)=f(z),\hspace{1cm} \frac{s_{\eta,\gamma}'(z)}{r_{\eta,\gamma}'(z)}=e^{i\theta} z^n, \theta \in \mathbb{R}, n \in \mathbb{N},
\end{equation}
where, $$zf'(z)= \frac{z}{(1+ze^{i(\eta + \gamma)})(1+ze^{-i(\eta - \gamma)})}$$ for some $\eta \in \mathbb{R}.$ As the analytic mappings $z/(1-ze^{i\alpha}),\alpha \in [0,2\pi)$ and $f$ are both convex, so $f_{b,\alpha}$ and $t_{\eta,\gamma}$ are in $S_H$ (see \cite{4}). Now we present the following result. \\

\begin{thm} Let $f_{b,\alpha}=h_{b,\alpha} +\overline{ g_{b,\alpha}}\in S_{H}$ be given by \eqref{eq:11} and $t_{\eta,\gamma}=r_{\eta,\gamma}+\overline{s_{\eta,\gamma}} \in S_H$ be given by \eqref{eq:12}. Then $f_{b,\alpha} \tilde{\ast} t_{\eta,\gamma} \in S_H$ and is convex in the direction of $-(\alpha+\gamma)$ for all $b\in \left(-1,\frac{-(n-2)}{n+2}\right].$
\end{thm} 
\begin{proof} By using \textit{\lq shear technique \rq} (see \cite{4}), from \eqref{eq:11},  we get $$h_{b,\alpha}(z)=\frac{1}{2}\left[\left(\frac{1+b}{1-b}\right)\frac{z}{(1-ze^{i\alpha})^2}+\frac{z}{(1-ze^{i\alpha})}\right]$$ and $$g_{b,\alpha}(z)=\frac{1}{2}\left[\left(\frac{1+b}{1-b}\right)\frac{ze^{2i\alpha}}{(1-ze^{i\alpha})^2}-\frac{ze^{2i\alpha}}{(1-ze^{i\alpha})}\right].$$ If $W$ denotes the dilatation function of $f_{b,\alpha}\tilde{\ast} t_{\eta,\gamma}=h_{b,\alpha} \ast r_{\eta,\gamma}+\overline{g_{b,\alpha} \ast s_{\eta,\gamma}},$ we get $$W(z)=\displaystyle\frac{ (g_{b,\alpha} \ast s_{\eta,\gamma})'(z)}{(h_{b,\alpha} \ast r_{\eta,\gamma})'(z)}.$$ In view of Lemma 3.1, we only need to show that $|W(z)|<1$ in $\mathbb{D}.$ As $$W(z)=e^{2i\alpha} \displaystyle\frac{( g_{b,0} \ast s_{\eta,\gamma})'(ze^{i\alpha})}{(h_{b,0} \ast r_{\eta,\gamma})'(ze^{i\alpha})}=e^{2i\alpha}\widehat{w}(ze^{i\alpha})(say),$$ it is therefore enough to prove that $|\widehat{w}(z)|<1$ in $\mathbb{D}.$ Further, note that 
\begin{equation} \label{eq:13}
\widehat{w}(z)=\frac{2bs_{\eta,\gamma}'(z)+(1+b)zs_{\eta,\gamma}''(z)}{2r_{\eta,\gamma}'(z)+(1+b)zr_{\eta,\gamma}''(z)}.
\end{equation}
From \eqref{eq:12}, we have $$s_{\eta,\gamma}'(z)=e^{i\theta}z^n r_{\eta,\gamma}'(z)$$ and therefore $$s_{\eta,\gamma}''(z)=e^{i\theta}z^n r_{\eta,\gamma}''(z)+n e^{i\theta}z^{n-1} r_{\eta,\gamma}'(z).$$ Putting these values of $s_{\eta,\gamma}'$ and $s_{\eta,\gamma}''$ in equation \eqref{eq:13}, we have \begin{equation} \label{eq:14}
\widehat{w}(z)=e^{i\theta}z^n\displaystyle \left[\frac{\frac{n+(n+2)b}{1+b}+\frac{zr_{\eta,\gamma}''(z)}{r_{\eta,\gamma}'(z)}}{\frac{2}{1+b}+\frac{zr_{\eta,\gamma}''z)}{r_{\eta,\gamma}'(z)}}\right].
\end{equation}
So, it is enough to prove that \begin{equation} \label{eq:15}
\displaystyle\left|\frac{\frac{n+(n+2)b}{1+b}+\frac{zr_{\eta,\gamma}''(z)}{r_{\eta,\gamma}'(z)}}{\frac{2}{1+b}+\frac{zr_{\eta,\gamma}''(z)}{r_{\eta,\gamma}'(z)}}\right|\leq 1. \end{equation}
We get equality in \eqref{eq:15} for $b=-\frac{n-2}{n+2}$ and for $b\in \left(-1,-\frac{n-2}{n+2}\right),$ $$\left(\frac{2}{1+b}\right)-\left(\frac{n+(n+2)b}{1+b}\right)>0.$$ Therefore,  in view of Lemma 2.1, it is sufficient to prove that
 $$Re\left\{\frac{zr_{\eta,\gamma}''(z)}{r_{\eta,\gamma}'(z)}\right\}>-\frac{n+2}{2}$$ in $\mathbb{D},$ or equivalently \begin{equation} \label{eq:15.1}
 Re\left\{n+2+\frac{2zr_{\eta,\gamma}''(z)}{r_{\eta,\gamma}'(z)}\right\}>0
 \end{equation} in $\mathbb{D}.$
 From \eqref{eq:12}, we have $$r_{\eta,\gamma}'(z)=\frac{f'(z)}{1-e^{i(\theta-2\gamma)}z^n},$$  which gives $$\frac{zr_{\eta,\gamma}''(z)}{r_{\eta,\gamma}'(z)}=\frac{zf''(z)}{f'(z)}+\frac{ne^{i(\theta-2\gamma)}z^n}{1-e^{i(\theta-2\gamma)}z^n}$$
 or $$n+2+\frac{2zr_{\eta,\gamma}''(z)}{r_{\eta,\gamma}'(z)}=2\left(1+\frac{zf''(z)}{f'(z)}\right)+n\left(\frac{1+e^{i(\theta-2\gamma)}z^n}{1-e^{i(\theta-2\gamma)}z^n}\right).$$ Now, $Re\left(1+\frac{zf''(z)}{f'(z)}\right)>0$ in $\mathbb{D}$ because $f$ is convex in $\mathbb{D}$ (as $zf'$ is starlike in $\mathbb{D}$) and for $|z|<1,$ we have $Re\left(\frac{1+e^{i(\theta-2\gamma)}z^n}{1-e^{i(\theta-2\gamma)}z^n}\right)>0.$ Thus \eqref{eq:15.1} is true.\end{proof}
 \begin{rem}
 We remark that Theorem 1.9 can be obtained from Theorem 3.2 by setting $b=\alpha=\gamma=0$ and choosing $\eta$ in $[\pi/2, \pi).$  \end{rem}
\section{Acknowledgement} First author is thankful to UGC - CSIR, New Delhi for providing financial support in the form of Senior Research Fellowship vide Award Letter Number - 2061540842.

\end{document}